\documentclass[10 pt, letterpaper]{article}
\usepackage[T1]{fontenc}
\usepackage{times}

\usepackage{xspace,eurosym,verbatim,enumitem}
\usepackage{amsmath,amssymb,bm,upgreek}
\usepackage{float,graphicx,epsfig,pstricks,pst-node,auto-pst-pdf,subfig}
\usepackage{ifpdf,cite}
\usepackage{multirow}

\psset{unit=1mm,framesep=10pt,fillstyle=none}  
\degrees[360]                                  
\psset{linewidth=0.9pt}                        
\providecommand{\mel}{\psset{linewidth=0.5pt}} 
\providecommand{\bol}{\psset{linewidth=1.1pt}} 

\newtheorem{theorem}{Theorem}
\newtheorem{lemma}[theorem]{Lemma}

\newtheorem{definition}[theorem]{Definition}
\newtheorem{remark}[theorem]{Remark}
\newtheorem{assumption}[theorem]{Assumption}

\newenvironment{proof}{\noindent\textit{Proof.}\rmfamily}{\hfill $\blacksquare$\vspace{2ex}}

\providecommand{\eg}{e.g.\@\xspace}
\providecommand{\ie}{i.e.\@\xspace}

\providecommand{\iid}{i.i.d.\@\xspace}
\providecommand{\st}{\text{s.t.}\enspace}
\providecommand{\pst}{\phantom{\text{s.t.}}\enspace}
\providecommand{\ef}{\enspace.}
\providecommand{\ec}{\enspace,}

\providecommand{\de}{\delta}
\providecommand{\ep}{\varepsilon}

\providecommand{\to}{\rightarrow}
\providecommand{\fa}{\forall\,\,} %
\providecommand{\tp}{^{\textrm{T}}} %
\providecommand{\co}{^{\textrm{C}}} %
\providecommand{\op}{^{\star}} %
\providecommand{\to}{\rightarrow}
\providecommand{\dm}{\,d}

\DeclareMathOperator{\lims}{lim\,sup} %
\DeclareMathOperator{\rnk}{rank}

\DeclareMathOperator{\Pb}{\mathbf{P}}
\DeclareMathOperator{\E}{\mathbf{E}}

\DeclareMathOperator{\If}{\mathbf{1}}
\DeclareMathOperator{\Id}{I}
\DeclareMathOperator{\B}{B}

\providecommand{\om}{\omega}
\providecommand{\Om}{\Omega}

\providecommand{\Vm}{V}
\providecommand{\Ab}{\bar{A}}
\providecommand{\Bb}{\bar{B}}
\providecommand{\Qb}{\bar{Q}}
\providecommand{\Phib}{\bar{\Phi}}

\providecommand{\CA}{\mathcal{A}}

\providecommand{\CF}{\mathcal{F}}

\providecommand{\CL}{\mathcal{L}}
\providecommand{\CN}{\mathcal{N}}
\providecommand{\CU}{\mathcal{U}}
\providecommand{\BN}{\mathbb{N}}
\providecommand{\BR}{\mathbb{R}}
\providecommand{\BRn}{\mathbb{R}^{n}}
\providecommand{\BRm}{\mathbb{R}^{m}}
\providecommand{\BU}{\mathbb{U}}

\providecommand{\BX}{\mathbb{X}}

\title{The Scenario Approach for Stochastic Model Predictive Control with Bounds on Closed-Loop Constraint Violations\thanks{This manuscript is the preprint of a paper submitted to Automatica and it is subject to Elsevier copyright. Elsevier maintains the sole rights of distribution or publication of the work in all forms and media. If accepted, the copy of record will be available at {\tt http://www.journals.elsevier.com/automatica/}.}}

\author{Georg Schildbach\thanks{Automatic Control Laboratory, Swiss Federal Institute of Technology Zurich (Physikstrasse 3, 8092 Zurich, Switzerland) {\tt schildbach|fagiano|morari@control.ee.ethz.ch}}
\and Lorenzo Fagiano\thanks{ABB Switzerland Ltd., Corporate Research Center (Segelhofstrasse 1, Baden-Daettwil, Switzerland) {\tt fagiano@control.ee.ethz.ch}}
 \and Christoph Frei\thanks{Mathematical and Statistical Sciences, University of Alberta (Edmonton, AB T6G 2G1, Canada) {\tt cfrei@ualberta.ca}} \and Manfred Morari$^{\dagger}$
}

\begin{document}

\maketitle



\begin{abstract}


Many practical applications of control require that constraints on the inputs and states of the system be respected, while optimizing some performance criterion. In the presence of model uncertainties or disturbances, for many control applications it suffices to keep the state constraints at least for a prescribed share of the time, as \eg in building climate control or load mitigation for wind turbines.
For such systems, a new control method of Scenario-Based Model Predictive Control (SCMPC) is presented in this paper. It optimizes the control inputs over a finite horizon, subject to robust constraint satisfaction under a finite number of random scenarios of the uncertainty and/or disturbances.
While previous approaches have shown to be conservative (\ie to stay far below the specified rate of constraint violations), the new method is the first to account for the special structure of the MPC problem in order to significantly reduce the number of scenarios. In combination with a new framework for interpreting the probabilistic constraints as average-in-time, rather than pointwise-in-time, the conservatism is eliminated.
The presented method retains the essential advantages of SCMPC, namely the reduced computational complexity and the handling of arbitrary probability distributions. It also allows for adopting sample-and-remove strategies, in order to trade performance against computational complexity.

\end{abstract}

\section{Introduction}\label{Sec:Intro}

Model Predictive Control (MPC) is a powerful approach for handling multi-variable control problems with constraints on the states and inputs. Its feedback control law can also incorporate feedforward information, \eg about the future course of references and/or disturbances, and the optimization of a performance criterion of interest.

Over the past two decades, the theory of linear and robust MPC has matured considerably \cite{Mayne:2000}. There are also widespread practical applications in diverse fields \cite{QinBadg:2003}. Yet many potentials of MPC are still not fully uncovered.

One active line of research is Stochastic MPC (SMPC), where the system dynamics are of a stochastic nature. They may be affected by additive disturbances \cite{BatinaEtAl:2002,CannonEtAl:2011,ChatEtAl:2011,CinqEtAl:2011,KouvEtAl:2010,LiEtAl:2002}, by random uncertainty in the system matrices \cite{CannonEtAl:2009a}, or both \cite{CannonEtAl:2009b,MunozEtAl:2005,PrimSung:2009,SchwNik:1999}. In this framework, a common objective is to minimize a cost function, while the system state is subject to chance constraints, \ie constraints that have to be satisfied only with a given probability.

Stochastic systems with chance constraints arise naturally in some applications, such as building climate control \cite{OldeEtAl:2012}, wind turbine control \cite{CannonEtAl:2009b}, or network traffic control \cite{YanBit:2005}. Alternatively, they can be considered as relaxations of robust control problems, in which the robust satisfaction of state constraints can be traded for an improved cost performance.

A major challenge in SMPC is the solution to chance-constrained finite-horizon optimal control problems (FHOCPs) in each sample time step. These correspond to non-convex stochastic programs, for which finding an exact solution is computationally intractable, except for very special cases \cite{KallMay:2011,Shapiro:2009}. Moreover, due to the multi-stage nature of these problems, it generally involves the computation of multi-variate convolution integrals \cite{CannonEtAl:2011}.

In order to obtain a tractable solution, various sample-based approximation approaches have been considered, \eg \cite{Batina:2004,BlackEtAl:2010,SkafBoyd:2009}. They share the significant advantage of coping with generic probability distributions, as long as a sufficient number of random samples (or `scenarios') can be obtained. The open-loop control laws can be approximated by sums of basis functions, as in the Q-design procedure proposed by \cite{SkafBoyd:2009}. However, these early approaches of Scenario-Based MPC (SCMPC) remain computationally demanding \cite{Batina:2004} and/or of a heuristic nature, \ie without specific guarantees on the satisfaction of the chance constraints \cite{BlackEtAl:2010,SkafBoyd:2009}.

More recent approaches \cite{CalFag:2013a,CalFag:2013b,Matusko:2012,PrandEtAl:2012,Schildi:2012,VayaEtAl:2012} are based on advances in the field of scenario-based optimization. However, these approaches share the drawback of being \emph{conservative} when applied in a receding horizon fashion, \ie the focus is either on obtaining a robust solution \cite{CalFag:2013a,CalFag:2013b,VayaEtAl:2012} or the chance constraints are over-satisfied by the closed loop system \cite{Matusko:2012,PrandEtAl:2012,Schildi:2012}. 

This conservatism of SCMPC represents a major practical issue, that is resolved by the contributions of this paper. In contrast to the previous results, the novel approach interprets the chance constraints as a time average, rather than pointwise-in-time with a high confidence, which is much less restrictive. Furthermore, the sample size is reduced by exploiting the structural properties of the finite-horizon optimal control problem \cite{SchildEtAl:2014}. The approach also allows for the presence of multiple simultaneous chance constraints on the state, and an a-posteriori removal of adverse samples for improving the controller performance \cite{Matusko:2012}.

In the most general setting, this paper considers linear systems with stochastic additive disturbances and uncertainty in the system matrices, which may only be known through a sufficient number of random samples. The computational complexity can be traded against performance of the controller by removing samples a-posteriori, starting from a simple convex linear or quadratic program and converging to the optimal SMPC solution in the limit.

The paper is organized as follows: Section \ref{Sec:ProbStat} presents a rigorous formulation of the optimal control problem that one would like to solve; Section \ref{Sec:SCMPC} describes how an approximated solution is obtained by SCMPC; Section \ref{Sec:Samples} develops the theoretical details, including the technical background and closed-loop properties; Section \ref{Sec:Exam} demonstrates the application of the method to a numerical example; and Section \ref{Sec:Conc} presents the main conclusions.

\section{Optimal Control Problem}\label{Sec:ProbStat}

Consider a discrete-time control system with a linear stochastic transition map
\begin{equation}\label{Equ:DynSystem}
	x_{t+1}=A(\de_{t})x_{t}+B(\de_{t})u_{t}+w(\de_{t})\ec\quad x_{0}=\bar{x}_{0}\ec
\end{equation}
for some fixed initial condition $\bar{x}_{0}\in\BRn$. The \emph{system matrix} $A(\de_{t})\in\BR^{n\times n}$ and the \emph{input matrix} $B(\de_{t})\in\BR^{n\times m}$ as well as the additive disturbance $w(\de_{t})\in\BRn$ are random, as they are (known) functions of a primal uncertainty $\de_{t}$. For notational simplicity, $\de_{t}$ comprises all uncertain influences on the system at time $t$.

\begin{assumption}[Uncertainty]\label{Ass:Uncertainty}
	(a) The uncertainties $\{\de_{0},\de_{1},...\}$, are independent and identically distributed
	(i.i.d.) random variables on a probability space $(\Delta,\Pb)$.
	(b) A `sufficient number' of \iid samples from $\de_{t}$ can be obtained, either empirically or
	by a random number generator.
\end{assumption}

The support set $\Delta$ of $\de_{t}$ and the probability measure $\Pb$ on $\Delta$ are entirely generic. In fact, $\Delta$ and $\Pb$ need not be known explicitly. The `sufficient number' of samples, which is required instead, will become concrete in later sections of the paper. Note that any issues arising from the definition of a $\sigma$-algebra on $(\Delta,\Pb)$ are glossed over in this paper, as they are unnecessarily technical. Instead, every relevant subset of $\Delta$ is assumed to be measurable.

The system \eqref{Equ:DynSystem} can be controlled by inputs $\{u_{0},u_{1},...\}$, to be chosen from a set of feasible inputs $\BU\subset\BRm$. Since the future evolution of the system \eqref{Equ:DynSystem} is uncertain, it is generally impractical to indicate all future inputs explicitly. Instead, each $u_{t}$ should be determined by a static feedback law
\begin{equation*}
    \psi:\BRn\to\BU\qquad\text{with}\qquad u_{t}=\psi(x_{t})\ec
\end{equation*}
based only on the current state of the system.

The optimal state feedback law $\psi$ should be determined in order to minimize the time-average of expected stage costs $\ell:\BRn\times\BRm\to\BR_{0+}$,
\begin{equation}\label{Equ:AvgCost}
    \frac{1}{T}\sum_{t=0}^{T-1}\E\bigl[\ell\bigl(x_{t},u_{t}\bigr)\bigr]\ef
\end{equation}
Each stage cost is taken in expectation $\E\bigl[\,\cdot\,\bigr]$, since its arguments $x_{t}$ and $u_{t}$ are random variables, being functions of $\{\de_{0},...,\de_{t-1}\}$. The time horizon $T$ is considered to be very large, yet it may not be precisely known at the point of the controller design.

The minimization of the cost is subject to keeping the state inside a state constraint set $\BX$ for a given fraction of all time steps. For many applications, the robust satisfaction of the state constraint (\ie $x_{t}\in\BX$ at all times $t$) is too restrictive for the choice of $\psi$, and results in a poor performance in terms of the cost function. This is especially true in cases where the lowest values of the cost function are achieved close to the boundary of $\BX$.  Moreover, it may be impossible to enforce if the support of $w(\de_{t})$ is unknown and possibly unbounded.

In order to make this more precise, let $M_{t}:=\If_{\BX\co}(x_{t+1})$ denote the random variable indicating that $x_{t+1}\notin\BX$, \ie $\If_{\BX\co}:\BRn\to\{0,1\}$ is the indicator function on the complement $\BX\co$ of $\BX$. The expected time-average of constraint violations should be upper bounded by some $\ep\in(0,0.5)$,
\begin{equation}\label{Equ:AvgViol}
    \E\bigl[\frac{1}{T}\sum_{t=0}^{T-1}M_{t}\bigr]\leq\ep\ef
\end{equation}

\begin{assumption}[Control Problem]\label{Ass:Control}
	(a) The state of the system can be measured at each time step $t$.
	(b) The set of \emph{feasible inputs} $\BU$ is bounded and convex.
	(c) The \emph{state constrained set} $\BX$ is convex.
	(d) The stage cost $\ell(\cdot,\cdot)$ is a convex function.
\end{assumption}

Assumption \ref{Ass:Control}(b) holds for most practical applications, and very large artificial bounds can always be introduced for input channels without natural bounds. Typical choices for the stage cost $\ell$ include
\begin{subequations}\label{Equ:StageCost}\begin{align}
    &\ell(\xi,\upsilon):=\bigl\|Q_\ell\xi\bigr\|_{1}+\bigl\|R_\ell\upsilon\bigr\|_{1}\ec\\
    \text{or}\quad
    &\ell(\xi,\upsilon):=\bigl\|Q_\ell\xi\bigr\|_\infty+\bigl\|R_\ell\upsilon\bigr\|_\infty\ec\\
    \text{or}\quad
    &\ell(\xi,\upsilon):=\bigl\|Q_\ell\xi\bigr\|_{2}^{2}+\bigl\|R_\ell\upsilon\bigr\|_{2}^{2}\ec
\end{align}\end{subequations}
where $Q_\ell\in\BR^{n\times n}$ and $R_\ell\in\BR^{m\times m}$ are positive semi-definite weighting matrices. Typical choices for the constraints $\BU$ and $\BX$ are polytopic or ellipsoidal sets.

Combining the previous discussions, the \emph{optimal control problem (OCP)} can be stated as follows:
\begin{subequations}\label{Equ:OCP}
    \begin{align}
    \min_{\psi}\quad&\frac{1}{T}\sum_{t=0}^{T-1}\E\bigl[\ell\bigl(x_{t},u_{t}\bigr)\bigr]\ec\\
    \st\quad& x_{t+1}=A(\de_{t})x_{t}+B(\de_{t})u_{t}+w(\de_{t})\ec\enspace x_{0}=\bar{x}_{0}
	    \quad \fa t=0,...,T-1\ec\\
    \pst\quad&\E\bigl[\frac{1}{T}\sum_{t=0}^{T-1}\If_{\BX\co}(x_{t})\bigr]\leq\ep\ec\\
    \pst\quad&\hspace*{0.1cm}u_{t}=\psi(x_{t})\quad\fa t=0,...,T-1\ef\hspace*{1.0cm}
	\end{align}
\end{subequations}
The equality constraints (\ref{Equ:OCP}b) are understood to be substituted recursively to eliminate all state variables $x_{0},x_{1},...,x_{T-1}$ from the problem. Thus only the state feedback law $\psi$ remains as a free variable in \eqref{Equ:OCP}.

\begin{remark}[Alternative Formulations]\label{Rem:Formulation}
    (a) Instead of the sum of expected values, the cost function (\ref{Equ:OCP}a) can also be
    defined as a desired quantile of the sum of discounted stage costs. Then the problem formulation
    corresponds to a minimization of the `value-at-risk', see \eg \cite{Shapiro:2009}.
    (b) Multiple chance constraints on the state $\BX_{j}$, each with an individual probability
    level $\ep_{j}$, can be included without further complications. A single chance constraint is
    considered here for notational simplicity.
\end{remark}

Many practical control problems can be cast in the general form of \eqref{Equ:OCP}. For example in building climate control \cite{OldeEtAl:2012}, the energy consumption of a building should be minimized, while its internal climate is subject to uncertain weather conditions and the occupancy of the building. The comfort range for the room temperatures may occasionally be violated without major harm to the system. Another example is wind turbine control \cite{CannonEtAl:2009b}, where the power efficiency of a wind turbine should be maximized, while its dynamics are subject to uncertain wind conditions. High stress levels in the blades must not occur too often, in order to achieve a desired fatigue life of the turbine.

\section{Scenario-Based Model Predictive Control}\label{Sec:SCMPC}

The OCP is generally intractable, as it involves an infinite-dimensional decision variable $\psi$ (the state feedback law) and a large number of constraints (growing with $T$). Therefore it is common to approximate it by various approaches, such as \emph{Model Predictive Control (MPC)}. 

\subsection{Stochastic Model Predictive Control (SMPC)}

The basic concept of MPC is to solve a tractable counterpart of \eqref{Equ:OCP} over a small horizon $N$ repeatedly at each time step. Only the first input of this solution is applied to the system \eqref{Equ:DynSystem}. In Stochastic MPC (SMPC), a \emph{Finite Horizon Optimal Control Problem (FHOCP)} is formulated by introducing chance constraints on the state:
\begin{subequations}\label{Equ:FHOCP}
    \begin{align}
    \min_{u_{0|t},...,u_{N-1|t}}\quad &\sum_{t=0}^{N-1}
	    \E\bigl[\ell\bigl(x_{i|t},u_{i|t}\bigr)\bigr]\ec\\
        \st\quad & x_{i+1|t}=A(\de_{t+i})x_{i|t}+B(\de_{t+i})u_{i|t}+w(\de_{t+i})\ec\,\,
	    x_{0|t}=x_{t}\quad\fa i=0,...,N-1\ec\\
        \pst\quad&\hspace*{0.24cm}\Pb\bigl[x_{i+1|t}\notin\BX\bigr]\leq\ep_{i}\quad\fa i=0,...,N-1\ec\\
        \pst\quad&\hspace*{0.32cm}u_{i|t}\in\BU\quad\fa i=0,...,N-1\ef
    \end{align}
\end{subequations}
Here $x_{i|t}$ and $u_{i|t}$ denote predictions and plans of the state and input variables made at time $t$, for $i$ steps into the future. The current measured state $x_{t}$ is introduced as an initial condition for the dynamics. The predicted states $x_{1|t},...,x_{N|t}$ are understood to be eliminated by recursive substitution of (\ref{Equ:FHOCP}b). Note that the predicted states are random by the influence of the uncertainties $\de_{t},...,\de_{t+N-1}$.

The \emph{probability levels} $\ep_{i}$ in the \emph{chance constraints} (\ref{Equ:FHOCP}c) usually coincide with $\ep$ from the OCP \cite{CinqEtAl:2011,OldeEtAl:2012,SchwNik:1999}, but they may generally differ \cite{YanBit:2005}. Some formulations also involve chance constraints over the entire horizon \cite{LiEtAl:2002,CannonEtAl:2009b}, or as a combination with robust constraints \cite{KouvEtAl:2010,CannonEtAl:2011}. Other alternatives of SMPC consider integrated chance constraints \cite{ChatEtAl:2011}, or constraints on the expectation of the state \cite{PrimSung:2009}.

\begin{remark}[Terminal Cost]\label{Rem:TermCost}
    An optional (convex) terminal cost $\ell_{f}:\BRn\to\BR_{0+}$ can be included in the FHOCP 
    \cite{Macie:2002,Mayne:2009}. In this case the term
    \begin{equation*}
        \E\bigl[\ell_{f}\bigl(x_{N|t}\bigr)\bigr]
    \end{equation*}
    would be added to the cost function (\ref{Equ:FHOCP}a).
\end{remark}

The state feedback law provided by SMPC is given by a receding horizon policy: the current state $x_{t}$ is substituted into (\ref{Equ:FHOCP}b), then the FHOCP is solved for an input sequence $\{u\op_{0|t},...,u\op_{N-1|t}\}$, and the current input is set to $u_{t}:=u\op_{0|t}$. This means that the FHOCP must be solved online at each time step $t$, using the current measurement of the state $x_{t}$.

However, the FHOCP is a stochastic program that remains difficult to solve, except for very special cases. In particular, the feasible set described by chance constraints is generally non-convex, despite of the convexity of $\BX$, and hard to determine explicitly. Hence a further approximation shall be made by scenario-based optimization.

\subsection{Scenario-Based Model Predictive Control (SCMPC)}

The basic idea of Scenario-Based MPC (SCMPC) is to compute an optimal finite-horizon input trajectory $\{u'_{0|t},...,u'_{N-1|t}\}$ that is feasible under $K$ of sampled `scenarios' of the uncertainty. Clearly, the scenario number $K$ has to be selected carefully in order to attain the desired properties of the controller. In this section, the basic setup of SCMPC is discussed, while the selection of a value for $K$ is deferred until Section \ref{Sec:Samples}.

More concretely, let $\de^{(1)}_{i|t},...,\de^{(K)}_{i|t}$ be \iid samples of $\de_{t+i}$, drawn at time $t\in\BN$ for the prediction steps $i=0,...,N-1$. For convenience, they are combined into \emph{full-horizon samples} $\om^{(k)}_{t}:=\{\de^{(k)}_{0|t},...,\de^{(k)}_{N-1|t}\}$, also called \emph{scenarios}. The \emph{Finite-Horizon Scenario Program (FHSCP)} then reads as follows:
\begin{subequations}\label{Equ:FHSCP}
	\begin{align}
	         \min_{u_{0|t},...,u_{N-1|t}}\quad &
         \sum_{k=1}^{K}\sum_{i=0}^{N-1}\ell\bigl(x^{(k)}_{i|t},u_{i|t}\bigr)\ec\\
	\st\quad & x^{(k)}_{i+1|t}=A(\de^{(k)}_{i|t})x^{(k)}_{i|t}+B(\de^{(k)}_{i|t})u_{i|t}
        +w(\de^{(k)}_{i|t})\c\,\, x^{(k)}_{0|t}=x_{t}\quad\fa i=0,...,N-1,\; k=1,...,K,\\
        \pst\quad & x^{(k)}_{i+1|t}\in\BX\;\;\fa i=1,...,N-1,\; k=1,...,K,\\
        \pst\quad & u_{i|t}\in\BU\;\;\fa i=0,...,N-1\ef
\end{align}\end{subequations}
The dynamics (\ref{Equ:FHSCP}b) provide $K$ different state trajectories over the prediction horizon, each corresponding to one sequence of affine transition maps defined by a particular scenario $\om^{(k)}_{t}$. Note that these $K$ state trajectories are not fixed, as they are still subject to the inputs $u_{0|t},...,u_{N-1|t}$. The cost function (\ref{Equ:FHSCP}a) approximates (\ref{Equ:FHOCP}a) as an average over all $K$ scenarios. The state constraints (\ref{Equ:FHSCP}c) are required to hold for $K$ sampled state trajectories over the prediction horizon.

Applying a receding horizon policy, the SCMPC feedback law is defined as follows (see also Figure \ref{Fig:Algorithm}, for $R=0$). At each time step $t\in\BN$ the current state measurement $x_{t}$ is substituted into (\ref{Equ:FHSCP}b), and the current input $u_{t}:=u'_{0|t}$ is set to the first of the optimal FHSCP solution $\{u'_{0|t},...,u'_{N-1|t}\}$, which is called the \emph{scenario solution}.

Unlike many MPC approaches, SCMPC does not have an inherent guarantee of \emph{recursive feasibility}, in the sense of \cite[Sec.\,4]{Mayne:2000}. Hence for a proper analysis of the closed-loop system, the following is assumed.

\begin{assumption}[Resolvability]\label{Ass:Resolvability}
    Under the SCMPC regime, each FHSCP admits a feasible solution at every time step $t$ almost
    surely.
\end{assumption}

While Assumption \ref{Ass:Resolvability} appears to be restrictive from a theoretical point of view, it is often reasonable from a practical point of view. For some applications, such as buildings \cite{OldeEtAl:2012}, recursive feasibility may hold by intuition, or it may be ensured by the use of \emph{soft constraints} \cite[Sec.\,2]{QinBadg:2003}. All in all, MPC remains a useful tool in practice, even for difficult stochastic systems \eqref{Equ:DynSystem} without the possibility of an explicit guarantee of recursive feasibility.

The following are possible alternatives and also convex formulations of \eqref{Equ:FHSCP}. The reasoning in each case is based on the theory in \cite{SchildEtAl:2014} and omitted for brevity.

\begin{remark}[Alternative Formulations]
    (a) Instead of the average cost in (\ref{Equ:FHSCP}a), the minimization may concern the cost of
    a nominal trajectory, as \eg in \cite{Schildi:2012,PrandEtAl:2012}; or the average may be taken
    over any sample size other than $K$.
    (b) The inclusion of additional chance constraints into \eqref{Equ:FHSCP}, as mentioned in
    Remark \ref{Rem:Formulation}(b), is straightforward. The number of scenarios $K_{j}$ may
    generally differ between multiple chance constraints.
    (c) In case of a value-at-risk formulation, as in Remark \ref{Rem:Formulation}(a), the average
    cost in (\ref{Equ:FHSCP}a) is replaced by the maximum:
    \vspace*{-0.3cm}
    \begin{equation*}
        \text{``}\sum_{k=1}^{K}\text{''}\qquad\longrightarrow\qquad
        \text{``}\max_{k=1,...,K}\text{''}\ec\vspace*{-0.3cm}
    \end{equation*}    
    where the sample size $K$ must be selected according to the desired risk level.
\end{remark}

\begin{remark}[Control Parameterization]\label{Rem:ContrParam}
    In the FHSCP, the predicted control inputs $u_{0|t},...,u_{N-1|t}$ may also be parameterized as
    a weighted sum of basis functions of the uncertainty, as proposed in \cite{SkafBoyd:2009,
    VayaEtAl:2012}.
    In particular, let $e_{1},...,e_{m}$ be the $J_{0}:=m$ unit vectors in $\BRm$, and for each time
    step $i=1,...,N$ let $q^{(j)}_{i|t}:\Delta^{i-1}\to\BRm$ be a finite set $j\in\{1,...,J_{i}\}$
    of pre-selected basis functions. Then
    \begin{align*}
    	&u_{0|t} := \sum_{j=1}^{J_{0}}\phi_{0}^{(j)}b_{j}\,,\\
    	&u_{i|t} := \sum_{j=1}^{J_{i}}\phi_{i}^{(j)}
	           q^{(j)}_{i|t}\bigl(\de^{(k)}_{0|t},...,\de^{(k)}_{i-1|t}\bigr)\quad\fa i=1,...,N-1\,,
    \end{align*}
    can be substituted into problem \eqref{Equ:FHSCP}, so that the weights $\phi_{i}^{(j)}\in\BR$
    for $i=0,...,N-1$ and $j=1,...,J_{i}$ become the new decision variables.
\end{remark}

A control parameterization with an increasing number of basis functions $J_{1},...,J_{N-1}$ generally improves the quality of the SCMPC feedback, while increasing the number of decision variables and hence the computational complexity; see \cite{SkafBoyd:2009,VayaEtAl:2012} for more details.

Given the sampled scenarios, \eqref{Equ:FHSCP} is a convex optimization program for which efficient solution algorithms exist, depending on its structure \cite{BoydVan:2004}. In particular, if $\BX$ and $\BU$ are polytopic (respectively ellipsoidal) sets, then the FHSCP has linear (second-order cone) constraints. If the stage cost is either (\ref{Equ:StageCost}a,b), then the FHSCP has a reformulation with a linear objective function, using auxiliary variables. If the stage cost is (\ref{Equ:StageCost}c), then the FHSCP can be expressed as a quadratic program. More details on these formulation procedures are found in \cite[pp.\,154\,f.]{Macie:2002}.

\subsection{A-Posteriori Scenario Removal}

A key merit of SCMPC is that it renders the uncertain control system (\ref{Equ:FHOCP}b) into multiple deterministic affine systems (\ref{Equ:FHSCP}b) by substituting particular scenarios. This significantly simplifies the solution to the FHSCP, as compared to the FHOCP. However, by introducing these random scenarios, a randomizing element is added to the SCMPC feedback law. In particular, the closed-loop system may occasionally show an erratic behavior due to highly unlikely outliers in the sampled scenarios. 

This effect can be mitigated by a-posteriori scenario removal, see \cite{CampGar:2011}. This allows for the \emph{state constraints} (\ref{Equ:FHSCP}c) corresponding to $R>0$ scenarios to be removed \emph{after} the outcomes of all samples have been observed. In exchange, the original sample size $K$ must be (appropriately) increased over its value for $R=0$. Any appropriate combination $(K,R)$ is called a \emph{sample-removal pair}. The choice of appropriate values for $K$ and $R$ is deferred to Section \ref{Sec:Samples}. The selection of removed scenarios is performed by a \emph{(scenario) removal algorithm} \cite[Def.\,2.1]{CampGar:2011}.

\begin{definition}[Removal Algorithm]\label{Def:RemAlg}
    (a) For each $\xi\in\BRn$, the \emph{(scenario) removal algorithm} $\CA_{\xi}:\Delta^{NK}\to
    \Delta^{N(K-R)}$ is a deterministic function selecting $(K-R)$ out of $K$ scenarios 
    $\{\om^{(1)}_{t},...,\om^{(K)}_{t}\}$. 
    (b) The selected scenarios at time step $t$ shall be denoted by
    \begin{equation*}
        \Om_{t}:=\CA_{x_{t}}\bigl(\om^{(1)}_{t},...,\om^{(K)}_{t}\bigr)\ef
    \end{equation*}
\end{definition}

Definition \ref{Def:RemAlg} is very general, in the sense that it covers a great variety of possible scenario removal algorithms. However, the most common and practical algorithms are described below:
\begin{description}
\item[\emph{\textmd{Optimal Removal:}}] The FHSCP is solved for all possible combinations of choosing $R$ out of $K$ scenarios. Then the combination that yields the lowest cost function value of all the solutions is selected. This requires the solution to $K$ choose $R$ instances of the FHSCP, a complexity that is usually prohibitive for larger values of $R$.\vspace*{0.2cm}
\item[\emph{\textmd{Greedy Removal:}}] The FHSCP is first solved with all $K$ scenarios. Then, in each of $R$ consecutive steps, the state constraints of a single scenario are removed that yields the biggest improvement, either in the total cost or in the first stage cost. Thus the procedure terminates after solving $KR-R(R-1)/2$ instances of FHSCP.\vspace*{0.2cm}
\item[\emph{\textmd{Marginal Removal:}}] The FHSCP is first solved with the state constraints of all $K$ scenarios. Then, in each of $R$ consecutive steps, the state constraints of a single scenario are removed based on the highest Lagrange multiplier. Hence the procedure requires the solution to $K$ instances of FHSCP.
\end{description}

Figure \ref{Fig:Algorithm} depicts an algorithmic overview of SCMPC, for the general case with scenario removal $R>0$. For the case without scenario removal, consider $R=0$ and the selected scenarios $\Om_{t}:=\{\om^{(1)}_{t},...,\om^{(K)}_{t}\}$. \vspace*{0.3cm}

\begin{figure}[h]
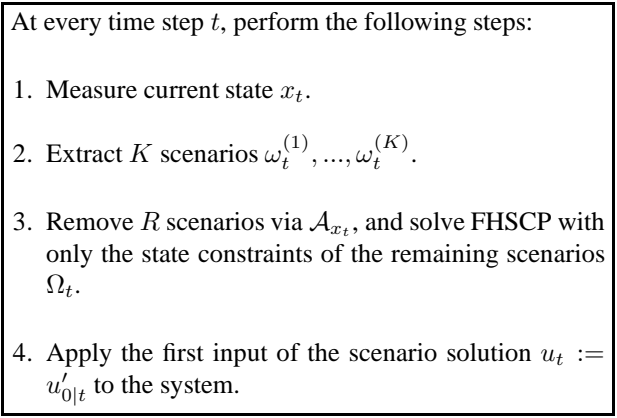

	\centering
	\setlength{\fboxrule}{0.5pt}
	\fbox{\begin{minipage}[h]{79mm}
	At every time step $t$, perform the following steps:\vspace*{0.2cm}
	\begin{enumerate}[leftmargin=0.47cm]
	\item Measure current state $x_{t}$.\vspace*{0.15cm}
	\item Extract $K$ scenarios $\om^{(1)}_{t},...,\om^{(K)}_{t}$.\vspace*{0.2cm}
	\item Remove $R$ scenarios via $\CA_{x_{t}}$, and solve FHSCP with only the state constraints of
	the remaining scenarios $\Om_{t}$.\vspace*{0.2cm}
	\item Apply the first input of the scenario solution $u_{t}:=u'_{0|t}$ to the system.
	\end{enumerate}
	\end{minipage}}
    \caption{Schematic overview of the SCMPC algorithm, for the case with scenario removal ($R>0$)
    and without scenario removal ($R=0$).\label{Fig:Algorithm}}
\end{figure}

\section{Problem Structure and Sample Complexity}\label{Sec:Samples}

For the SCMPC algorithm described in Section \ref{Sec:SCMPC}, the sample-removal pair $(K,R)$ remains to be specified. Appropriate values for $K$ and $R$ are theoretically derived in this section. Their values generally depend on the control system and the constraints, and $K$ is referred to as the \emph{sample complexity} of the SCMPC problem.

For some intuition about this problem, suppose that $R\geq 0$ is fixed and the sample size $K$ is increased. This means that the solution to the FHSCP becomes robust to more scenarios, with the following consequences. First, the average-in-time state constraint violations \eqref{Equ:AvgViol} decrease, in general. Therefore the state constraint will translate into a lower bound on $K$. Second, the computational complexity increases as well as the average-in-time closed-loop cost \eqref{Equ:AvgCost}, in general. Therefore the objective is to choose $K$ as small as possible, and ideally equal to its lower bound.

The higher the number of removed constraints $R\geq 0$, the higher will be the lower bound on $K$, in order for the state constraints \eqref{Equ:AvgViol} to be satisfied. Now consider pairs $(R,K)$ of removed constraints $R$ together with their corresponding lower bounds $K$, which equally satisfy the state constraints \eqref{Equ:AvgViol}. For the intuition, suppose $R$ is increased, so $K$ increases as well. Then the computational complexity grows, due to more constraints in the FHSCP and the removal algorithm. At the same time, the solution quality of the FHSCP improves, in general, and hence the average-in-time closed-loop cost \eqref{Equ:AvgCost} decreases. Therefore $R$ is usually fixed to a value that is as high as admitted by the available computational resources.

\subsection{Support Rank}\label{Sec:SRank}

According to the classic scenario approach \cite{CampGar:2008,CampGar:2011}, the relevant quantity for determining the sample size $K$ for a single chance constraint (with a fixed $R$) is the number of \emph{support constraints} \cite[Def.\,2.1]{CampGar:2008}. In fact, $K$ grows with the (unknown) number of support constraints, so the goal is to obtain a tight upper bound. For the classic scenario approach, this upper bound is given by the dimension of the decision space \cite[Prop.\,2.2]{CampGar:2008}, \ie $Nm$ in the case of the FHSCP. 

The FHSCP is a multi-stage stochastic program, with multiple chance constraints (namely $N$, one per stage). This requires an extension to the classic scenario approach; the reader is referred to \cite{SchildEtAl:2014} for more details. Now each chance constraint contributes an individual number of support constraints, to which an upper bound must be obtained. These individual upper bounds are provided by the \emph{support rank} of each chance constraint \cite[Def.\,3.6]{SchildEtAl:2014}.

\begin{definition}[Support Rank]\label{Def:SuppRank}
    (a) The \emph{unconstrained subspace} $\CL_{i}$ of a constraint $i\in\{0,...,N-1\}$ in 
    (\ref{Equ:FHSCP}c) is the largest (in the set inclusion sense) linear subspace of the search
    space $\BR^{Nm}$ that remains unconstrained by all sampled instances of $i$, almost surely.
    (b) The \emph{support rank} of a constraint $i\in\{0,...,N-1\}$ in (\ref{Equ:FHSCP}c) is
    \begin{equation*} 
        \rho_{i}:=Nm-\dim\CL_{i}\ec
    \end{equation*}
    where $\dim\CL_{i}$ represents the dimension of the unconstrained subspace $\CL_{i}$. 
\end{definition}

Note that the support rank is an inherent property of a particular chance constraint and it is not affected by the simultaneous presence of other constraints. Hence the set of constraints of the FHSCP may change, for instance, due to the reformulations of Remark \ref{Rem:Formulation}. 

Besides the extension to multiple chance constraints, the support rank has the merit of a significant reduction of the upper bound on the number of support constraints. Indeed, the following two lemmas replace the classic upper bound $Nm$ with much lower values, such as $l\leq n$ or $m$, depending on the problem structure. 

For systems affected by \emph{additive} disturbances only, the support rank of any state constraint in the FHSCP is given by the support rank $l\leq n$ of $\BX$ in $\BRn$ (\ie the co-dimension of the largest linear subspace that is unconstrained by $\BX$).

\begin{lemma}[Pure Additive Disturbances]\label{The:AddDist}
    Let $l\leq n$ be the support rank of $\BX$ and suppose that $A\bigl(\de^{(k)}_{i|t}\bigr)\equiv
    A$ and $B\bigl(\de^{(k)}_{i|t}\bigr)\equiv B$ are constant and the control is not parameterized
    (as in Remark \ref{Rem:ContrParam}).
    Then the support rank of any state constraint $i\in\{0,...,N-1\}$ in (\ref{Equ:FHSCP}c) is at
    most $l$.
\end{lemma}

For systems affected by \emph{additive and multiplicative} disturbances, Lemma \ref{The:AddDist} no longer holds. However, it will be seen that for the desired closed-loop properties, the relevant quantity for selecting the sample size $K$ is the support rank $\rho_{1}$ of the state constraint on $x_{1|t}$ only. For this first predicted step, the support rank is restricted to at most $m$, under both additive and multiplicative disturbances.

\begin{lemma}[Additive and Multiplicative Disturbances]\label{The:FirstStep}
    The support rank $\rho_{1}$ of constraint $i=1$ in (\ref{Equ:FHSCP}c) is at most $m$.
\end{lemma}

For the sake of readability, the proofs of Lemmas \ref{The:AddDist} and \ref{The:FirstStep} are deferred to Appendix \ref{Sec:LemProof}. They effectively decouple the support rank, and hence the sample size $K$, from the horizon length $N$. 

Note that the result of Lemma \ref{The:FirstStep} holds also for the parameterized control laws of Remark \ref{Rem:ContrParam}. In this case, it decouples the sample size $K$ from the number of basis functions $J_{i}$ for all stages $i=1,...,N-1$.

Tighter bounds of $\rho_{1}$ than those in Lemmas \ref{The:AddDist} and \ref{The:FirstStep} may exist, resulting from a special structure of the system \eqref{Equ:DynSystem} and/or the state constraint set $\BX$. The basic insights to exploit this can be found in the Appendix \ref{Sec:LemProof} and \cite{SchildEtAl:2014}.

\subsection{Sample Complexity}\label{Sec:SampComp}

This section describes the selection of the sample-removal pair $(K,R)$, based on a bound of the support rank $\rho_{1}$. Throughout this subsection, the initial state $x_{t}$ is considered to be fixed to an arbitrary value. 

Let $\Vm_{t}|x_{t}$ denote the \emph{(first step) violation probability}, \ie the probability with which the first predicted state falls outside of $\BX$: 
\begin{equation}\label{Equ:DefViol}
    \Vm_{t}|x_{t}:=
	\Pb\bigl[A(\de_{t})x_{t}+B(\de_{t})u'_{0|t}+w(\de_{t})\notin\BX\,\big|\,x_{t}\bigr]\ef
\end{equation}
Recall that $u'_{0|t}$ denotes the first input of the scenario solution $\{u'_{0|t},...,u'_{N-1|t}\}$. Clearly, $u'_{0|t}$ and $\Vm_{t}|x_{t}$ depend on the scenarios $\Om_{t}$ that are substituted into the FHSCP at time $t$. The notation $u'_{0|t}(\Om_{t})$ and $\Vm_{t}|x_{t}(\Om_{t})$ shall be used occasionally to emphasize this fact. 

The violation probability $\Vm_{t}|x_{t}(\Om_{t})$ can be considered as a random variable on the probability space $(\Delta^{KN},\Pb^{KN})$, with support in $[0,1]$. Here $\Delta^{KN}$ and $\Pb^{KN}$ denote the $KN$-th product of the set $\Delta$ and the measure $\Pb$, respectively. For distinction, the expectation operator on $(\Delta,\Pb)$ is denoted $\E$, and that on $(\Delta^{KN},\Pb^{KN})$ is denoted $\E^{KN}$.

The distribution of $\Vm_{t}|x_{t}(\Om_{t})$ is unknown, being a complicated function of the entire control problem \eqref{Equ:FHOCP} and the removal algorithm $\CA_{x_{t}}$. However, it is possible to derive the following upper bound on this distribution.

\begin{lemma}[Upper Bound on Distribution]\label{The:DistrBound}
    Let Assumptions \ref{Ass:Uncertainty}, \ref{Ass:Control}, \ref{Ass:Resolvability} hold and 
    $x_{t}\in\BRn$ be an arbitrary initial state. For any violation level $\nu\in[0,1]$,
    \begin{subequations}
    \begin{equation}\label{Equ:DistrBound}
        \Pb^{KN}\bigl[\Vm_{t}|x_{t}(\Om_{t})>\nu\bigr]\leq U_{K,R,\rho_{1}}(\nu)\:,
        \hspace*{2.03cm}
    \end{equation}
    \begin{equation}
         U_{K,R,\rho_{1}}(\nu):=
        \min\Bigl\{1,{R+\rho_{1}-1\choose R}\B\bigl(\nu;K,R+\rho_{1}-1\bigl)\Bigr\}\:,
    \end{equation}
    \end{subequations}
    where $\B(\,\cdot\,;\,\cdot\,,\,\cdot\,)$ represents the beta distribution function 
    \cite[frm.\,26.5.3,\,26.5.7]{Abramowitz:1970},
    \begin{equation*}
        \B\bigl(\nu;K,R+\rho_{1}-1\bigl):=
        \sum_{j=0}^{R+\rho_{1}-1}{K\choose j}\nu^{j}(1-\nu)^{K-j}\:.
    \end{equation*}
\end{lemma}

\begin{proof}
    The proof is a straightforward extension of \cite[Thm.\,6.7]{SchildEtAl:2014}, where the bound
    on $\Vm_{t}|x_{t}(\Om_{t})$ is saturated at $1$.
\end{proof}

This paper exploits the result of Lemma \ref{The:DistrBound} to obtain an upper bound on the expectation
\begin{equation}
    \E^{KN}\bigl[\Vm_{t}\,\big|\,x_{t}\bigr]:=\int_{\Delta^{KN}}\Vm_{t}|x_{t}(\Om_{t})\dm\Pb^{KN}\ef
\end{equation}
A reformulation via the indicator function $\If:\Delta^{KN}\to\{0,1\}$ yields that
\begin{align}\label{Equ:NumInt}
    \E^{KN}\bigl[\Vm_{t}\,\big|\,x_{t}\bigr]
    &=\int_{[0,1]}\int_{\Delta^{KN}}\If\bigl(\Vm_{t}|x_{t}(\Om_{t})>\nu\bigr)\dm\Pb^{KN}d\nu
                                                                                         \nonumber\\
    &=\int_{[0,1]}\Pb^{KN}\bigl[\Vm_{t}|x_{t}(\Om_{t})>\nu\bigr]\dm\nu\nonumber\\
    &\leq\int_{[0,1]}U_{K,R,\rho_{1}}(\nu)\dm\nu\ef
\end{align}

\begin{definition}[Admissible Sample-Removal Pair]\label{Def:AdmPair}
    A sample-removal pair $(K,R)$ is \emph{admissible} if its substitution into \eqref{Equ:NumInt}
    yields $\E^{KN}\bigl[\Vm_{t}\,\big|\,x_{t}\bigr]\leq\ep$.
\end{definition}

Whether a given sample-removal pair $(K,R)$ is admissible can be tested by performing the one-dimensional numerical integration \eqref{Equ:NumInt}. It can easily be seen that the integral value \eqref{Equ:NumInt} monotonically decreases with $K$ and monotonically increases with $R$. Hence, if either $K$ or $R$ is fixed, an admissible sample-removal pair $(K,R)$ can be determined \eg by a bisection method. Moreover, if $R$ is fixed, there always exist $K$ large enough to generate an admissible pair $(K,R)$.

\begin{remark}[No Scenario Removal]\label{Rem:NoRemoval}
    If $R=0$, the integration \eqref{Equ:NumInt} can be replaced by the exact analytic formula
    \begin{equation}\label{Equ:SimpNumInt}
        \E^{KN}\bigl[\Vm_{t}\,\big|\,x_{t}\bigr]\leq\frac{\rho_{1}}{K+1}\ef
    \end{equation}
\end{remark}

Figure \ref{Fig:ViolProb} illustrates the monotonic relationship of the upper bound \eqref{Equ:NumInt} in $K$ and $R$. Supposing that $R=0,30,100$ is fixed, the corresponding admissible pair $(K,R)$ can be found by moving along the graphs until the desired violation level $\ep$ is reached. The solid and the dashed line correspond to different support dimensions $\rho_{1}=2$ and $\rho_{1}=5$.

\begin{figure}[h]
	\centering
    \parbox[c][74mm]{83mm}{\centering
    \begin{pspicture}(2,0)(85,74)
    \input{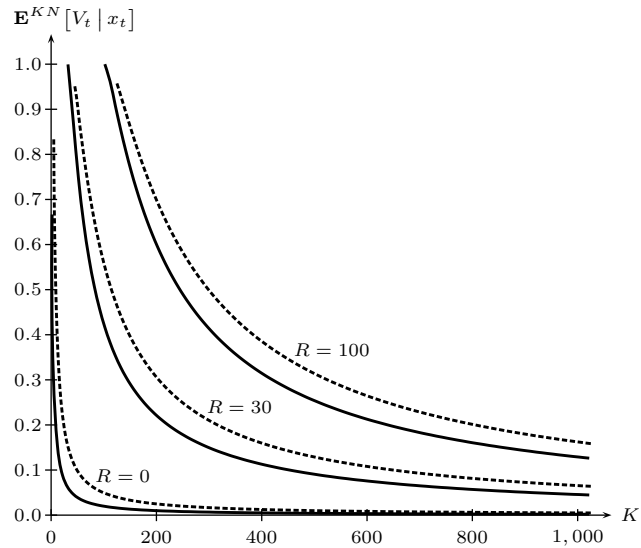}
    \end{pspicture}}
    \caption{Upper bound on the expected violation probability $\E^{KN}\bigl[\Vm_{t}\,\big|\,x_{t}
    \bigr]$, as a function of the sample size $K$, for different scenario removals $R$ and support
    dimensions $\rho_{1}=2$ (solid lines) and $\rho_{1}=5$ (dashed lines).
    \label{Fig:ViolProb}}
\end{figure}

\subsection{Closed-Loop Properties}

This section analyzes the closed-loop properties of the control system under the SCMPC law for an admissible sample-removal pair $(K,R)$. To this end, the underlying stochastic process is first described. Recall that
\begin{itemize}
\item $x_{0},...,x_{T-1}$ is the closed-loop trajectory, where $x_{t}$ depends on all past uncertainties $\de_{0},...,\de_{t-1}$ as well as all past scenarios $\Om_{0},...,\Om_{t-1}$;
\item $\Vm_{0},...,\Vm_{T-1}$ are the violation probabilities, where $\Vm_{t}$ depends on $x_{t}$ and $\Om_{t}$, and hence on $\Om_{0},...,\Om_{t}$ and $\de_{0},...,\de_{t-1}$;
\item $M_{0},...,M_{T-1}$ indicate the actual violation of the constraints, where $M_{t}$ depends on $x_{t+1}$, and hence on $\Om_{0},...,\Om_{t}$ and $\de_{0},...,\de_{t}$.
\end{itemize}

At each time step $t$, there are a total of $D:=(KN+1)$ random variables, namely the scenarios together with the disturbance $\{\de_{t},\Om_{t}\}\in\Delta^{(KN+1)}=\Delta^{D}$. In order to simplify notations, define
\begin{equation*}
   \CF_{t}:=\{\de_{0},\Om_{0},...,\de_{t},\Om_{t}\}\in\Delta^{(t+1)D}\ec
\end{equation*}
for any $t\in\{0,...,T-1\}$. These auxiliary variables allow for the random variables $x_{t}(\CF_{t-1})$, $\Vm_{t}(\CF_{t-1},\Om_{t})$, $M_{t}(\CF_{t})$ to be expressed in terms of their elementary uncertainties. Moreover, let $\Pb^{(t+1)D}$ denote the probability measure and $\E^{(t+1)D}$ the expectation operator on $\Delta^{(t+1)D}$, for any $t\in\{0,...,T-1\}$.

Observe that $M_{t}\in\{0,1\}$ is a Bernoulli random variable with (random) parameter $\Vm_{t}$, because
\begin{align}\label{Equ:BinVar}
    \E\bigl[M_{t}\,\big|\,\CF_{t-1},\Om_{t}\bigr]
                                             &=\int_{\Delta}M_{t}(\CF_{t})\dm\Pb(\de_{t})\nonumber\\
                                             &=\Vm_{t}(\CF_{t-1},\Om_{t})
\end{align}
for any values of $\CF_{t-1},\Om_{t}$.

\begin{theorem}\label{The:ConvExp}
    Let Assumptions \ref{Ass:Uncertainty}, \ref{Ass:Control}, \ref{Ass:Resolvability} hold and
    $(K,R)$ be an admissible sample-removal pair.
    Then the expected time-average of closed-loop constraint violations \eqref{Equ:AvgViol} remains
    below the specified level $\ep$,
    \vspace*{-0.3cm}
    \begin{equation}\label{Equ:AsymExp}
    	\E^{TD}\bigl[\frac{1}{T}\sum_{t=0}^{T-1}M_{t}\bigr]\leq\ep\ef\vspace*{-0.3cm}
	\end{equation}
    for any $T\in\BN$.
\end{theorem}

\begin{proof}
By linearity of the expectation operator,
\begin{align*}
	&\E^{TD}\bigl[\frac{1}{T}\bigl(M_{0}+M_{1}+...+M_{T-1}\bigr)\bigr]\\
	&\,\,=\frac{1}{T}\bigl(\E^{D}\bigl[M_{0}\bigr]+\E^{2D}\bigl[M_{1}\bigr]+...+\E^{TD}\bigl[M_{T-1}\bigr]\bigr)\nonumber\\
	&\,\,=\frac{1}{T}\bigl(\E^{D-1}\bigl[\Vm_{0}\bigr]+\E^{2D-1}\bigl[\Vm_{1}\bigr]+...+\E^{TD-1}\bigl[\Vm_{T-1}\bigr]\bigr)\,,
\end{align*}
by virtue of \eqref{Equ:BinVar}. Moreover, for any $t\in\{0,...,T-1\}$,
\begin{equation*}
	\E^{(t+1)D-1}\bigl[\Vm_{t}\bigr]
	=\int_{\Delta^{tD}}\underbrace{\E^{D-1}\bigl[\Vm_{t}\,\big|\,\CF_{t-1}\bigr]}_{\leq\ep}
	\dm\Pb^{tD}\leq\ep\,,
\end{equation*}
where the integrand is pointwise upper bounded by $\ep$ because $(K,R)$ is an admissible sample-removal pair.
\end{proof}

Theorem \ref{The:ConvExp} shows that the chance constraints of the OCP can be expected to be satisfied over any finite time horizon $T$. The next Lemma \ref{The:AsymProb} sets the stage for an  even stronger result, Theorem \ref{The:ConvSure}, showing that the chance constraint are satisfied almost surely as $T\to\infty$.

\begin{lemma}\label{The:AsymProb}
    If Assumptions \ref{Ass:Uncertainty}, \ref{Ass:Control}, \ref{Ass:Resolvability} hold, then
    \vspace*{-0.3cm}
	\begin{equation}\label{Equ:AsymProb}
		\lim_{T\to\infty}\frac{1}{T}\sum_{t=0}^{T-1}
		\Bigl(M_{t}-\E^{D-1}\bigl[\Vm_{t}\big|\CF_{t-1}\bigr]\Bigr)=0\vspace*{-0.3cm}
	\end{equation}
	almost surely.
\end{lemma}

\begin{proof}
For any $t\in\BN$, define $Z_{t}:=M_{t}-\E^{D-1}\bigl[\Vm_{t}\big|\CF_{t-1}\bigr]$ and observe that
\begin{align}\label{Equ:AsymProb1}
    &\E^{D}\bigl[Z_{t}\big|\CF_{t-1}\bigr]\\
    &\,\,\,=\E^{D}\bigl[M_{t}\big|\CF_{t-1}\bigr]-\E^{D}\bigl[\E^{D-1}\bigl[\Vm_{t}\big|
                                                      \CF_{t-1}\bigr]\big|\CF_{t-1}\bigr]\nonumber\\
    &\,\,\,=\E^{D}\bigl[M_{t}\big|\CF_{t-1}\bigr]-\E^{D-1}\bigl[\Vm_{t}\big|\CF_{t-1}\bigr]
                                                                                         \nonumber\\
    &\,\,\,=0\ec
\end{align}
by virtue of \eqref{Equ:BinVar}. In probabilistic terms, this says that $\{Z_{t}\}_{t\in\BN}$ is a sequence of martingale differences. Moreover, 
\begin{equation}\label{Equ:AsymProb2}
    \sum_{t=0}^{\infty}\frac{1}{(t+1)^2}\E^{D}\bigl[Z_{t}^{2}\big|\CF_{t-1}\bigr]<\infty
\end{equation}
almost surely, because $|Z_{t}|\leq 1$ is bounded for $t\in\BN$. Therefore \cite[Thm.\,2.17]{HallHeyde:1980} can be applied, which yields that
\begin{equation}\label{Equ:AsymProb2}
    \sum_{t=0}^{T-1}\frac{1}{t+1}Z_{t}
\end{equation}
converges almost surely as $T\to\infty$. The result \eqref{Equ:AsymProb} now follows by use of Kronecker's Lemma, \cite[p.\,31]{HallHeyde:1980}.
\end{proof}

Note that Lemma \ref{The:AsymProb} does not imply that
\begin{equation}\label{Equ:AsymProb3}
    \lim_{T\to\infty}\frac{1}{T}\sum_{t=0}^{T-1}M_{t}=
    \lim_{T\to\infty}\frac{1}{T}\sum_{t=0}^{T-1}\E^{D-1}\bigl[\Vm_{t}\big|\CF_{t-1}\bigr]
\end{equation}
almost surely, because it is not clear that the right-hand side converges almost surely. However, if it converges almost surely, then \eqref{Equ:AsymProb3} holds. 

\begin{theorem}\label{The:ConvSure}
    Let Assumptions \ref{Ass:Uncertainty}, \ref{Ass:Control}, \ref{Ass:Resolvability} hold and
    $(K,R)$ be an admissible sample-removal pair. Then
    \vspace*{-0.3cm}
	\begin{equation}\label{Equ:ConvSure}
		\underset{T\to\infty}{\lims}\,\frac{1}{T}\sum_{t=0}^{T-1}M_{t}\leq\ep\vspace*{-0.3cm}
	\end{equation}
	almost surely.
\end{theorem}

\begin{proof}
    From Lemma \ref{The:AsymProb},
    \begin{align}\label{Equ:ConvSure1}
		0=&\lim_{T\to\infty}\frac{1}{T}\sum_{t=0}^{T-1}
		\Bigl(M_{t}-\E^{D-1}\bigl[\Vm_{t}\big|\CF_{t-1}\bigr]\Bigr)\nonumber\\
		\geq&\underset{T\to\infty}{\lims}\,\frac{1}{T}\sum_{t=0}^{T-1}\bigl(M_t-\ep\bigr)\nonumber\\
		=&\underset{T\to\infty}{\lims}\,\frac{1}{T}\sum_{t=0}^{T-1}M_{t}-\ep
	\end{align}
	almost surely, where the second line follows from Definition \ref{Def:AdmPair}.
\end{proof}

\section{Numerical Example}\label{Sec:Exam}

\subsection{System Data}

Consider the stochastic linear system
\begin{equation*}\label{Equ:ExSystem}
	x_{t+1}=
	\begin{bmatrix}
		0.7 &-0.1(2+\theta_t) \\ -0.1(3+2\theta_t) &0.9
	\end{bmatrix}
	x_{t}+
	\begin{bmatrix}
		1 &0 \\ 0 &1
	\end{bmatrix}
	u_{t}+
	\begin{bmatrix}
		w_{t}^{(1)}\\
		w_{t}^{(2)}
	\end{bmatrix},
\end{equation*}
where $x_{0}=[1\; 1]\tp$. Here $\theta_{t}\sim\CU\bigl([0,1]\bigr)$ is uniformly distributed on the interval $[0,1]$ and $w_{t}^{(1)},w_{t}^{(2)}\sim\CN(0,0.1)$ are normally distributed with mean $0$ and variance $0.1$. The inputs are confined to
\begin{equation*}
    \BU:=\bigl\{\upsilon\in\BR^{2}\,\big|\,|\upsilon^{(1)}|\leq 5\,\wedge\,|\upsilon^{(2)}|\leq 5\bigr\},
\end{equation*}
and two state constraints are considered:
\begin{equation*}
    \BX_{1}:=\bigl\{\xi\in\BR^{2}\,\big|\,\xi^{(1)}\geq 1\bigr\}\ec\enspace
    \BX_{2}:=\bigl\{\xi\in\BR^{2}\,\big|\,\xi^{(2)}\geq 1\bigr\},
\end{equation*}
either individually or in combination $\BX:=\BX_{1}\cap\BX_{2}$. The stage cost function is chosen to be of the quadratic form (\ref{Equ:StageCost}c), with the weights $Q_\ell:=I$ and $R_\ell:=I$. The MPC horizon is set to $N:=5$.

\subsection{Joint Chance Constraint}

The support rank of the joint chance constraint $\BX$ is bounded by $\rho_{1}=2$. Figure \ref{Fig:PhaseSingle} depicts a phase plot of the closed-loop system trajectory, for two admissible sample-removal pairs (a) $(19,0)$ and (b) $(1295,100)$, corresponding to $\ep=10\%$. Instances in which the state trajectory leaves $\BX$ are indicated in red. Note that the distributions are centered around a similar mean in both cases, however the case $R=0$ features stronger outliers than  $R=100$.

\begin{figure}[h]
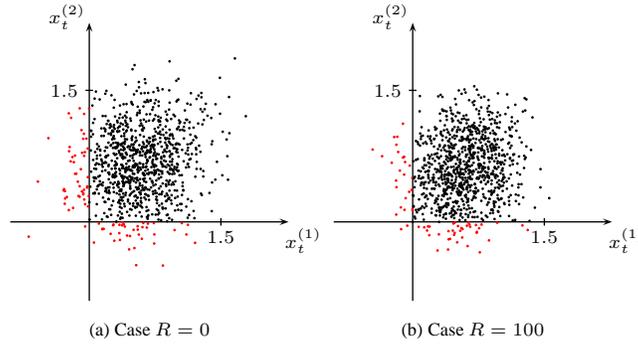

	\vspace*{0.3cm}
	\centering
    \parbox[c][40mm]{83mm}{\centering
    \begin{pspicture}(0,0)(83,40)
        \definecolor{purple}{cmyk}{0.45,0.55,0,0}
        \scriptsize
        \mel
        \rput[tc](18.5,-3){(a) Case $R=0$}\rput[tc](61.5,-3){(b) Case $R=100$}
        \psline{->}(10.5,1)(10.5,38)\psline{-}(28,11)(28,12)
        \psline{->}(53.5,1)(53.5,38)\psline{-}(71,11)(71,12)
        \rput[tc](28,9.5){$1.5$}\rput[tc](71,9.5){$1.5$} 
        \rput[tr](10,40.5){$x_{t}^{(2)}$}\rput[tr](53,40.5){$x_{t}^{(2)}$}
        \psline{->}(0,11.5)(37,11.5)\psline{-}(10,29)(11,29)
        \psline{->}(43,11.5)(80,11.5)\psline{-}(53,29)(54,29)
        \rput[mr](9,29){$1.5$}\rput[mr](52,29){$1.5$} 
        \rput[lc](39,9){$x_{t}^{(1)}$}\rput[lc](82,9){$x_{t}^{(1)}$}
        \input{./SR000vio.tex}
        \input{./SR000sat.tex}
        \input{./SR100vio.tex}
        \input{./SR100sat.tex}
    \end{pspicture}}
    \vspace*{0.3cm}
    \caption{Phase plot of closed-loop system trajectory (red: violating states; black: other
    states). The axis lines mark the boundary of the feasible set $\BX$.\label{Fig:PhaseSingle}}
\end{figure}

Table \ref{Tab:Joint} shows the empirical results of a simulation of the closed-loop system over $T=10,000$ time steps. Note that there is essentially no conservatism in the case of no removals ($R=0$). Some minor conservatism is present for small removal sizes, disappearing asymptotically as $R\to\infty$. At the same time, the reduction of the average closed-loop cost $\ell_{\textrm{avg}}$ is minor for this example, while the standard deviation $\ell_{\textrm{std}}$ is affected significantly.

\renewcommand\arraystretch{1.5}
\begin{table}[H]
	\begin{center}
    \begin{tabular}{c||rrrr}
       $\ep=10\%$ & $R=0$ & $R=50$ & $R=100$ & $R=500$\\ \hline\hline
       K & $19$ & $702$ & $1,295$ & $5,723$ \\
       $V_{\textrm{avg}}$ & $9.87\%$ & $7.37\%$ & $8.06\%$ & $8.74\%$ \\
       $\ell_{\textrm{avg}}$ & 3.78  & 3.75 & 3.72 & 3.68  \\
       $\ell_{\textrm{std}}$ & 0.54 & 0.44 & 0.42 & 0.37
    \end{tabular}
    \end{center}
	\vspace*{0.1cm}
    \caption{Joint chance constraint: closed-loop results for mean violations $V_{\textrm{avg}}$,
    mean stage cost $\ell_{\textrm{avg}}$, and standard deviation of stage costs
    $\ell_{\textrm{std}}$.\label{Tab:Joint}}
\end{table}
\renewcommand\arraystretch{1.0}

To highlight the impact of the presented SCMPC approach, the results of Table \ref{Tab:Joint} can be compared to those of previous SCMPC approaches \cite{Schildi:2012,CalFag:2013a}. The sample size is $19$ (compared to about $400$), and the empirical share of constraint violations in closed-loop is $9.87\%$ (compared to about $0.05\%$). These figures become even worse when longer horizons are considered; \eg for $N=20$, previous approaches require about 900 samples and yield about $0.2\%$ violations.

\subsection{Individual Chance Constraints}

For the same example, the two chance constraints $\BX_{1}$ and $\BX_{2}$ are now considered separately, with the individual probability levels $\ep_{1}=5\%$ and $\ep_{2}=10\%$. Each support rank is bounded by $\rho_{1}=1$. Figure \ref{Fig:PhaseDouble} depicts a phase plot of the closed-loop system trajectory, for the admissible sample-removal pairs (a) $(19,0)$, $(9,0)$ and (b) $(2020,100)$, $(1010,100)$.

\begin{figure}[h]
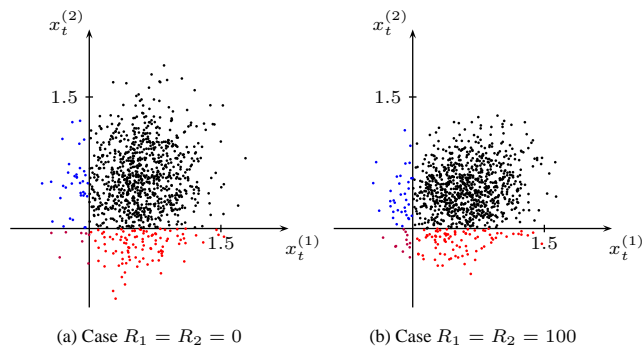

	\vspace*{0.3cm}
	\centering
    \parbox[c][40mm]{83mm}{\centering
    \begin{pspicture}(0,0)(83,40)
        \scriptsize
        \mel
        \rput[tc](18.5,-3){(a) Case $R_1=R_2=0$}\rput[tc](61.5,-3){(b) Case $R_1=R_2=100$}
        \psline{->}(10.5,1)(10.5,38)\psline{-}(28,11)(28,12)
        \psline{->}(53.5,1)(53.5,38)\psline{-}(71,11)(71,12)
        \rput[tc](28,9.5){$1.5$}\rput[tc](71,9.5){$1.5$} 
        \rput[tr](10,40.5){$x_{t}^{(2)}$}\rput[tr](53,40.5){$x_{t}^{(2)}$}
        \psline{->}(0,11.5)(37,11.5)\psline{-}(10,29)(11,29)
        \psline{->}(43,11.5)(80,11.5)\psline{-}(53,29)(54,29)
        \rput[mr](9,29){$1.5$}\rput[mr](52,29){$1.5$} 
        \rput[lc](39,9){$x_{t}^{(1)}$}\rput[lc](82,9){$x_{t}^{(1)}$}
        \input{./DR000vio1.tex}
        \input{./DR000vio2.tex}
        \input{./DR000vio12.tex}
        \input{./DR000sat.tex}
        \input{./DR100vio1.tex}
        \input{./DR100vio2.tex}
        \input{./DR100vio12.tex}
        \input{./DR100sat.tex}
    \end{pspicture}}
    \vspace*{0.3cm}
    \caption{Phase plot of closed-loop system trajectory (blue, red, purple: violating states of 
    $\BX_1$, $\BX_2$, $\BX_1$ and $\BX_2$; black: other states). The axis lines mark the boundaries
    of the feasible sets $\BX_1$ and $\BX_2$, respectively.\label{Fig:PhaseDouble}}
\end{figure}

Table \ref{Tab:Ind} shows the empirical results of a simulation of the closed-loop system over $T=10,000$ time steps. Note that there is very little conservatism in all cases. As in the previous example, the reduction of the average closed-loop cost $\ell_{\textrm{avg}}$ is minor, while the standard deviation $\ell_{\textrm{std}}$ is affected significantly.

\renewcommand\arraystretch{1.5}
\begin{table}[H]
	\begin{center}
    \begin{tabular}{c||rrr}
       $\ep_{1}=5\%,$ & $R_1=R_2$ & $R_1=R_2$ & $R_1=R_2$ \\
       $\ep_{2}=10\%$ & $=0$ & $=50$ & $=100$ \\ \hline\hline
       $K_1$ & $19$ & $1,020$ & $2,020$ \\
       $K_2$ & $9$ & $510$ & $1,010$ \\
       $V_{\textrm{avg},1}$ & $5.14\%$ & $4.84\%$ & $4.95\%$  \\
       $V_{\textrm{avg},2}$ & $9.94\%$ & $9.81\%$ & $9.93\%$  \\
       $\ell_{\textrm{avg}}$ & 3.67  & 3.62  & 3.51   \\
       $\ell_{\textrm{std}}$ & 0.54  & 0.46  & 0.42
    \end{tabular}
    \end{center}
	\vspace*{0.1cm}
    \caption{Single chance constraint: closed-loop results for mean violations $V_{\textrm{avg},1}$
    and $V_{\textrm{avg},2}$ of $\BX_{1}$ and $\BX_{2}$, mean stage cost $\ell_{\textrm{avg}}$, and
    standard deviation of stage costs $\ell_{\textrm{std}}$.
	\label{Tab:Ind}}
\end{table}
\renewcommand\arraystretch{1.0}

\section{Conclusion}\label{Sec:Conc}

The paper has presented new results on Scenario-Based Model Predictive Control (SCMPC). By focusing on the average-in-time probability of constraint violations and by exploiting the multi-stage structure of the finite-horizon optimal control problem (FHOCP), the number of scenarios has been greatly reduced compared to previous approaches. Moreover, the possibility to adopt a-posteriori constraint removal strategies is also accommodated. Due to its computational efficiency, the presented approach paves the way for a tractable application of Stochastic Model Predictive Control (SMPC) to large-scale problems with hundreds of decision variables.

\section*{Acknowledgements}

Research leading to these results has received funding from the European Union Seventh
Framework Programme FP7/2007-2013 under grant agreement number FP7-ICT-2009-4 248940, and
under grant agreement number PIOF-GA-2009-252284. Christoph Frei gratefully acknowledges
financial support by the Natural Sciences and Engineering Research Council of Canada.

\begin{appendix}
\section{Proof of Lemmas \ref{The:AddDist} and \ref{The:FirstStep}}\label{Sec:LemProof}

The particular bounding arguments follow rather easily after some general observations on the support rank. Pick any state constraint $i\in\{1,...,N\}$ from (\ref{Equ:FHSCP}c). Recursively substituting the dynamics (\ref{Equ:FHSCP}b), the constrained state can be expressed as
\begin{subequations}
	\begin{align}\label{Equ:StateCon}
	& x^{(k)}_{i|t}=\bigl(A^{(k)}_{i|t}\cdot ...\cdot A^{(k)}_{0|t}\bigr)x_{t}
    +\Ab^{(k)}_{i|t}\Bb^{(k)}_{i|t}
    \begin{bmatrix}
        u_{0|t}\\
        \vdots\\
        u_{N-1|t}
    \end{bmatrix}
    +\Ab^{(k)}_{i|t}
    \begin{bmatrix}
        w^{(k)}_{0|t}\\
        \vdots\\
        w^{(k)}_{i-1|t}
    \end{bmatrix},\\
    &\Ab^{(k)}_{i|t}:=
    \begin{bmatrix}
        A^{(k)}_{i|t}\cdot ...\cdot A^{(k)}_{1|t}\\
        \vdots\\ A^{(k)}_{1|t}\\ \Id
    \end{bmatrix}\tp,\\
    &\Bb^{(k)}_{i|t}:=
    \begin{bmatrix}
        B^{(k)}_{0|t} & 0 &\hdots &0 &0 &\hdots &0\\
        0 & B^{(k)}_{1|t} &\hdots &0 &0 &\hdots &0\\
        \vdots &\vdots &\ddots &\vdots &0 &\hdots &0\\
        0 & 0 &\hdots &B^{(k)}_{i|t} &0 &\hdots &0\\
    \end{bmatrix},
\end{align}\end{subequations}
where $I\in\BR^{n\times n}$ denotes the identity matrix, and for any $i=0,...,N-1$ the following abbreviations are used:
\begin{equation*}\label{Equ:AbbrVar}
    A^{(k)}_{i|t}:=A\bigl(\de^{(k)}_{i|t}\bigr),\enspace
    B^{(k)}_{i|t}:=B\bigl(\de^{(k)}_{i|t}\bigr),\enspace
    w^{(k)}_{i|t}:=w\bigl(\de^{(k)}_{i|t}\bigr)\ef
\end{equation*}
Let $l\leq n$ be the support rank of $\BX$, \ie the co-dimension of the largest linear subspace that is unconstrained by $\BX$. Then there exists a projection matrix $P\in\BR^{l\times n}$ such that for each $x\in\BRn$
\begin{equation*}
    x\in\BX\quad\Longleftrightarrow\quad Px\in P\BX:=\bigl\{P\xi\:\big|\:\xi\in\BX\bigr\}\ef
\end{equation*}
For example, if the state constraint concerns only the first two elements of the state vector, then $l=2$ and $P\in\BR^{2\times n}$ may contain the first two unit vectors $e_{1},e_{2}\in\BRn$ as its rows.

\subsection*{Proof of Lemma \ref{The:AddDist}}

If $A\bigl(\de^{(k)}_{i|t}\bigr)\equiv A$ and $B\bigl(\de^{(k)}_{i|t}\bigr)\equiv B$ are constant for all $i\in\{0,...,N-1\}$, then \eqref{Equ:StateCon} reduces to
\begin{equation}\label{Equ:AddDist}
    \underbrace{
    \begin{bmatrix}
        PA^{i-1}B &\hdots & P & 0 &\hdots
    \end{bmatrix}}_{\rnk(\cdot)\,\leq l}
    \begin{bmatrix}
        u_{0|t}\\
        \vdots\\
        u_{N-1|t}
    \end{bmatrix}
    +PA^{i}x_{t}+
    \begin{bmatrix}
        PA^{i-1}B &\hdots & P
    \end{bmatrix}
    \begin{bmatrix}
        w^{(k)}_{0|t}\\
        \vdots\\
        w^{(k)}_{i-1|t}
    \end{bmatrix}\in P\BX\ec
\end{equation}
for any $i\in\{1,...,N\}$. The rank of the first matrix of dimension $l\times Nm$ can be at most $l$, and therefore it has a null space of dimension at least $Nm-l$. The disturbance has no effect on this null space, because it enters only through the third, additive term in \eqref{Equ:AddDist}. Hence this null space is clearly an unconstrained subspace of the constraint and $\rho_{i}\leq l\leq n$ for all $i\in\{1,...,N\}$, proving Lemma \ref{The:AddDist}.

\subsection*{Proof of Lemma \ref{The:FirstStep}}

Consider the first state constraint $i=1$ of (\ref{Equ:FHSCP}c). Here \eqref{Equ:StateCon} reduces to
\begin{equation}\label{Equ:FirstState}
    \underbrace{
    \begin{bmatrix}
    	P\Bb^{(k)}_{0|t} &0 &\hdots &0
	\end{bmatrix}}_{\rnk(\cdot)\,\leq m}
	\begin{bmatrix}
        u_{0|t}\\
        \vdots\\
        u_{N-1|t}
    \end{bmatrix}
    +PA^{(k)}_{0|t}x_{t}+
    Pw^{(k)}_{0|t}\in P\BX\ef
\end{equation}
The rank of the first matrix can here be at most $m$ for all outcomes of $\Bb^{(k)}_{0|t}$, because the last $(N-1)m$ variables in the decision vector are always in its null space. Hence $\rho_{1}\leq m$ in all cases, proving Lemma \ref{The:FirstStep}.

\subsection*{Parameterized Control Laws}

For the case of parameterized control laws as in Remark \ref{Rem:ContrParam}, it will be shown that the argument of Lemma \ref{The:FirstStep} continues to apply. Define for any $i=1,...,N-1$
\begin{align*}
	&Q_{0|t}:=\Id\ec\qquad &\Phi_{0|t}:=\phi_{0|t}\ec\quad\,\,\,\,\\
	&Q_{i|t}^{(k)}:=
	\underbrace{
	\begin{bmatrix}
		q_{i|t}^{(1)} & q_{i|t}^{(2)} & \hdots & q_{i|t}^{(J_{i})} 
	\end{bmatrix}}_{\in\BR^{m\times J_{i}}}\ec\quad
	&\Phi_{i|t}:=
	\underbrace{
	\begin{bmatrix}
		\phi_{i|t}^{(1)}\\
		\vdots\\
		\phi_{i|t}^{(J_{i})}\\
	\end{bmatrix}}_{\in\BR^{J_{i}}}\ec
\end{align*}
where $q_{i|t}^{(j)}:=q_{i|t}^{(j)}\bigl(\de^{(k)}_{0|t},...,\de^{(k)}_{i|t}\bigr)$ is used as an abbreviation and $\Id\in\BR^{m\times m}$ denotes the identity matrix. Then the vector of control inputs under scenario $k=1,...,K$ can be expressed as the matrix-vector product
\begin{equation*}
	\begin{bmatrix}
        u_{0|t}\\
        u_{1|t}^{(k)}\\
        \vdots\\
        u_{N-1|t}^{(k)}
    \end{bmatrix}
    =
    \underbrace{
    \begin{bmatrix}
        Q_{0|t}    & 0             & \hdots & 0 \\
        0          & Q_{1|t}^{(k)} & \hdots & 0 \\
        \vdots     & \vdots        & \ddots & \vdots \\
        0          & 0             & \hdots & Q_{N-1|t}^{(k)}
    \end{bmatrix}}_{=:\Qb_{t}^{(k)}}
    \underbrace{
    \begin{bmatrix}
    	\Phi_{0|t}\\
        \Phi_{1|t} \\
        \vdots \\
        \Phi_{N-1|t}       
    \end{bmatrix}}_{=:\Phib_{t}}\ef
\end{equation*}
Substitute this, in place of the original decision vector in \eqref{Equ:FirstState} to see that the same rank argument as before applies.

\end{appendix}

\newpage
\bibliographystyle{plain}
\bibliography{contrbib,engbib,mathbib}

\end{document}